\documentclass[a4paper]
{amsart}
\usepackage{microtype}

\usepackage[shortlabels]{enumitem}
\setlist[enumerate]{label={\arabic*.}}
\setlist[description]{font=\normalfont\slshape}

\usepackage[dvipsnames]{xcolor}
\usepackage[backref]{hyperref}
\usepackage[abbrev,shortalphabetic]{amsrefs}
\hypersetup{colorlinks=true,linkcolor=Maroon,citecolor=OliveGreen}

\usepackage{cleveref}
\usepackage[norefs, nocites]{refcheck}

\makeatletter
\newcommand{\refcheckize}[1]{%
  \expandafter\let\csname @@\string#1\endcsname#1%
  \expandafter\DeclareRobustCommand\csname relax\string#1\endcsname[1]{%
    \csname @@\string#1\endcsname{##1}\wrtusdrf{##1}}%
  \expandafter\let\expandafter#1\csname relax\string#1\endcsname
}
\makeatother
\refcheckize{\cref}
\refcheckize{\Cref}

\usepackage{mathtools}
\usepackage{extarrows}

\usepackage{tikz-cd}

\newtheorem{theorem}{Theorem}[section]
\newtheorem{lemma}[theorem]{Lemma}

\newtheorem{corollary}[theorem]{Corollary}

\theoremstyle{definition}

\newtheorem{problem}[theorem]{Problem}

\newtheorem{example}[theorem]{Example}

\usepackage{amssymb}
\usepackage{mathtools}
\renewcommand\bar{\overline}
\renewcommand\hat{\widehat}

\newcommand{\eps}{\epsilon}
\renewcommand\subset{\subseteq}

\renewcommand\phi{\varphi}

\newcommand\br[1]{\left(#1\right)}

\def\bias{\opr{bias}}
\def\torus{\mathbf{T}}
\def\tr{\opr{tr}}  

\newcommand\opr[1]{\operatorname{#1}}

\def\R{\mathbf{R}}
\def\C{\mathbf{C}}
\def\Z{\mathbf{Z}}
\def\F{\mathbf{F}}

\def\P{\mathbf{P}}
\def\E{\mathbf{E}}

\def\AR{\opr{AR}}
\def\PR{\opr{PR}}

\def\CB{\opr{CB}}

\def\cod{\opr{cod}}

\usepackage{todonotes}

\begin{document}

\title{Biased multilinear maps of abelian groups}
\author{Sean Eberhard}

\thanks{S. Eberhard has received funding from the European Research Council (ERC) under the European Union’s Horizon 2020 research and innovation programme (grant agreement No. 803711).}

\begin{abstract}
    We adapt the theory of partition rank and analytic rank to the category of abelian groups.
    If $A_1, \dots, A_k$ are finite abelian groups and $\phi : A_1 \times \cdots \times A_k \to \torus$ is a multilinear map, where $\torus = \R/\Z$,
    the bias of $\phi$ is defined to be the average value of $\exp(i 2 \pi \phi)$.
    If the bias of $\phi$ is bounded away from zero we show that $\phi$ is the sum of boundedly many multilinear maps each of which factors through the standard multiplication map of $\Z/q\Z$ for some bounded prime power $q$.
    Relatedly, if $F : A_1 \times \cdots \times A_{k-1} \to B$ is a multilinear map such that $\P(F = 0)$ is bounded away from zero, we show that $F$ is the sum of boundedly many multilinear functions of a particular form.
    These structure theorems generalize work of several authors in the elementary abelian case to the arbitrary abelian case.
    The set of all possible biases is also investigated.
\end{abstract}

\maketitle

\tableofcontents

\section{Introduction}

Suppose $A_1, \dots, A_k$ are finite abelian groups and $\phi : A_1 \times \cdots \times A_k \to \torus$ is a multilinear map,
where $\torus = \R/\Z$.
Let $e(x) = \exp(i 2 \pi x)$ be the standard character of $\torus$.
The \emph{bias} of $\phi$ is defined by
\begin{equation}
  \label{bias-def}
  \bias(\phi) = \E_{x \in A_{[k]}} e(\phi(x)).
\end{equation}
Here and throughout we use the following index notation.
The symbol $[k]$ denotes the index set $\{1, \dots, k\}$.
For $I \subset [k]$,
\[
  A_I = \prod_{i \in I} A_i.
\]
For $x \in A_{[k]}$,
\[
  x_I = (x_i)_{i \in I} \in A_I.
\]
We write $I^c$ for $[k] \setminus I$.
Also, we are using the expectation symbol to denote the (finite) average over the set $A_{[k]}$.

The concept of bias often comes up in the following way.
Suppose $F : A_{[k-1]} \to B$ is a multilinear map of abelian groups
such that $\P(F = 0)$ is bounded away from zero.
Letting $A_k$ be the dual group $\hat B$, we can define a multilinear map $\phi : A_{[k]} \to \torus$ by
\[
    \phi(x) = x_k(F(x_{[k-1]})),
\]
and for this map we have $\bias(\phi) = \P(F = 0)$.
Moreover, $F$ is determined by $\phi$, so if we can say something about the structure of $\phi$
then we will know something about the structure of $F$.
Thus the study of multilinear kernels is reduced to the study of bias.

If $A_1, \dots, A_k$ are vector spaces over a field $\F = \F_q$ then it is usually more natural to consider a multi-$\F$-linear map
$\phi : A_{[k]} \to \F$ and to define
\[
    \bias(\phi) = \E_{x \in A_{[k]}} \chi(\phi(x))
\]
where $\chi : \F_q \to \torus$ is any standardized character such as $\chi(x) = e((\tr x) / p)$,
where $\tr : \F_q \to \F_p$ is the absolute trace.
Bias in this context was introduced by Gowers and Wolf in \cite{gowers-wolf} and studied extensively by several authors.
The term \emph{analytic rank} is used for the quantity
\[
    \AR(\phi) = \log_q (\bias(\phi)^{-1}).
\]
The \emph{partition rank} $\PR(\phi)$ of $\phi$ is the minimal number $r$ such that $\phi$ is the sum of $r$ multilinear maps of the form
\[
    \phi_1(x_I) \phi_2(x_{I^c}),
\]
where $I \subset [k]$ and $0 < |I| < k$.
Multilinear maps of bounded partition rank are uniformly biased: in fact, $\AR(\phi) \leq \PR(\phi)$.
This is one of several nice lemmas appearing in a paper by Lovett~\cite{lovett}.
The main theorem in this field is a converse: uniformly biased multilinear maps have bounded partition rank,
or in other words partition rank is bounded in terms of analytic rank.

\begin{theorem}
    \label{thm:AR-PR-field-case}
    Suppose $A_1, \dots, A_k$ are finite $\F_q$-vector spaces and $\phi : A_{[k]} \to \F_q$ is multilinear over $\F_q$.
    If $\bias(\phi) \geq \eps > 0$ then $\phi$ has partition rank $O_{\eps, k,q}(1)$.
    In other words, $\PR(\phi) \leq F(k, \AR(\phi), q)$ for some function $F$.
\end{theorem}

This result was essentially proved by Green and Tao~\cite{green-tao}.
Subsequent research has focused on quantitative aspects (which were initially poor).
The dependence on $q$ was removed by Bhowmick and Lovett in \cite{bhowmick-lovett} (so $\PR(\phi) \leq F(k, \AR(\phi))$).
In breakthrough work Janzer~\cite{janzer} and Mili\'cevi\'c~\cite{milicevic} independently proved that $\PR(\phi) \leq C_k(\AR(\phi)^{D_k} + 1)$ for constants $C_k, D_k$.
Recently, Cohen and Moshkovitz~\cite{cohen-moshkovitz} proved that $\PR(\phi) \leq (2^{k-1}+1) \AR(\phi) + 1$ provided that $q > q_0(k, \AR(\phi))$.

In this note we consider biased multilinear maps of arbitrary abelian groups.
We will use \Cref{thm:AR-PR-field-case} as a block box (and only the prime field case)
and we will deduce an analogous structure theorem for arbitrary abelian groups.
To state this structure theorem we need a suitable notion of partition rank.
For $q$ a prime power, let $m_q : (\Z/q\Z)^2 \to \torus$ be the bilinear map defined by $m(x, y) = xy / q \bmod 1$.
We say that $\phi : A_{[k]} \to \torus$ \emph{factors through} $m_q$ if it has the form
\[
    \phi(x) = m_q(\phi_1(x_I), \phi_2(x_{I^c}))
\]
for some $I \subset [k]$ with $0 < |I| < k$.
Here $\phi_1 : A_I \to \Z/q\Z$ and $\phi_2 : A_{I^c} \to \Z/q\Z$ must both be multilinear.

\begin{theorem}
    \label{thm:bias-inverse-theorem}
    Suppose $A_1, \dots, A_k$ are finite abelian groups and $\phi : A_{[k]} \to \torus$ is multilinear.
    If $\bias(\phi) \geq \eps > 0$ then $\phi$ is the sum of $O_{\eps, k}(1)$ multilinear maps each of which factors through $m_q$ for some prime power $q \leq O_{\eps, k}(1)$.
\end{theorem}

It is easy to see that \Cref{thm:AR-PR-field-case} and \Cref{thm:bias-inverse-theorem} agree in the case of vector spaces over a prime finite field,
but neither result is more general than the other.
(A common generalization would consider $R$-modules for some commutative ring $R$;
\Cref{thm:AR-PR-field-case} would be the case $R = \F_q$ and \Cref{thm:bias-inverse-theorem} would be the case $R = \Z$.)

The following corollary follows from the connection mentioned already between bias and multilinear kernels.
It states that a multilinear function $F$ with $\P(F = 0)$ bounded away from zero
must be the sum of boundedly many functions each of which
``crushes'' a nontrivial subset of the variables before mapping to the range.
Here $\cod(g)$ denotes the codomain of the function $g$.

\begin{corollary}
    \label{cor:main-corollary}
    Suppose $A_1, \dots, A_{k-1}, B$ are finite abelian groups and $F : A_{[k-1]} \to B$ is a multilinear map such that $\P(F = 0) \geq \eps > 0$.
    Then there is an expression
    \[
        F(x) = \sum_{\emptyset \neq I \subset [k-1]} G_I(g_I(x_I), x_{[k-1] \setminus I}),
    \]
    where for each $I$ the functions $g_I$ and $G_I$ are multilinear maps
    \begin{align*}
        &g_I : A_I \to \cod(g_I), \\
        &G_I : \cod(g_I) \times A_{[k-1] \setminus I} \to B,
    \end{align*}
    and $|\cod(g_I)| \leq O_{\eps,k}(1)$.
\end{corollary}

In the last section we give an application of \Cref{thm:bias-inverse-theorem} to the set all possible biases.
Let $\Phi_k$ be the set of all multilinear maps $\phi : A_1 \times \cdots \times A_k \to \torus$ (for any finite abelian groups $A_1, \dots, A_k$)
and let
\[
    B_k = \{ \bias(\phi) : \phi \in \Phi_k\}.
\]
We show that $B_k$ is a small subset of $[0, 1]$ in various senses; for example, all its limit points are algebraic.

Elsewhere we will give an application to probabilistically nilpotent finite groups.

\section{Basic properties of bias}

Several nice lemmas about bias over vector spaces were proved by Lovett in \cite{lovett}.
We need analogues in the category of abelian groups.

In this section we will sometimes consider functions $\phi : A_{[k]} \to \torus$ which are not multilinear,
but we still define $\bias(\phi)$ by \eqref{bias-def}.
If $\phi : A_{[k]} \to \torus$ is a map we write $\phi_{x_I} : A_{I^c} \to \torus$ for the map obtained by restriction:
\[
    \phi_{x_I}(x_{I^c}) = \phi(x).
\]
Obviously,
\begin{equation}
  \label{bias-recursion}
  \bias(\phi) = \E_{x_I} \bias(\phi_{x_I}).
\end{equation}
In particular, consider  \eqref{bias-recursion} in the case $I = \{i\}^c$.
If $\phi(x)$ is linear in $x_i$ then $\bias(\phi_{x_I})$ is 1 or 0 according to whether $\phi_{x_I} \equiv 0$,
so
\begin{equation}
  \label{bias-kernel-expression}
  \bias(\phi) = \P_{x_I} (\phi_{x_I} \equiv 0).
\end{equation}
This shows that, while in general $\bias(\phi)$ is complex-valued, $\bias(\phi) \in [0, 1]$ provided that $\phi$ is linear in at least one of its arguments.

\begin{lemma}
    \label{bias-trivial-bounds}
    Assume $\phi:A_{[k]} \to \torus$ is multilinear.
    For each $i \in [k]$,
    \[
      \bias(\phi) \geq 1 - \prod_{j \neq i} (1 - 1/|A_j|).
    \]
    If $\phi$ is nontrivial then
    \[
      \bias(\phi) \leq 1 - \prod_{j \neq i} (1 - 1/p_j),
    \]
    where $p_j$ is the smallest prime divisor of $|A_j|$.
\end{lemma}
\begin{proof}
    Let $I = \{i\}^c$. If $x_j = 0$ for any $j \in I$ then $\phi_{x_I} \equiv 0$, so from \eqref{bias-kernel-expression} we have
    \[
      1 - \bias(\phi) = \P(\phi_{x_I} \not\equiv 0) \leq \prod_{j\in I} \P(x_j \neq 0) = \prod_{j \neq i} (1 - 1/|A_j|).
    \]

    The second inequality is proved by induction.
    The case $k = 1$ is clear, because if $\phi$ is linear and nontrivial then $\bias(\phi) = 0$, and the right-hand side is also $0$.
    Let $k > 1$.
    By \eqref{bias-recursion} with $I = \{j\}$ we have
    \[
      1 - \bias(\phi) = \E_{x_j} (1 - \bias(\phi_{x_j})).
    \]
    Let $B_j = \{x_j \in A_j : \phi_{x_j} \equiv 0\}$.
    By induction if $x_j \notin B_j$ then
    \[
      1 - \bias(\phi_{x_j}) \geq \prod_{j' \neq i, j} (1-1/p_{j'}).
    \]
    Thus
    \[
      1 - \bias(\phi) \geq (1 - |B_j|/|A_j|) \prod_{j' \neq i, j} (1-1/p_{j'}).
    \]
    But $B_j$ is a subgroup of $A_j$, and proper since $\phi$ is nontrivial, so $|B_j| / |A_j| \leq 1/p_j$.
    This completes the induction.
\end{proof}

\begin{lemma}
    \label{main-term-bias-lemma}
    For $I \subset [k]$, let $\phi_I : A_I \to \torus$ be $|I|$-linear.
    Let $\phi : A_{[k]} \to \torus$ be the function
    \[
      \phi(x) = \sum_{I \subset [k]} \phi_I(x_I).
    \]
    Suppose $J \subset [k]$ is any subset such that $\phi_I = 0$ for $I \supsetneq J$.
    Then
    \[
      |\bias(\phi)| \leq \bias(\phi_J).
    \]
\end{lemma}
\begin{proof}
Let $i \in [k]$.
Write
\[
  \phi = \phi_{i} + \phi_{i'},
\]
where
\begin{align*}
  \phi_i &= \sum_{I \subset [k], i \in I} \phi_I, \\
  \phi_{i'} &= \sum_{I \subset [k], i \notin I} \phi_I.
\end{align*}
Note that $\phi_i$ is linear in $x_i$ and $\phi_{i'}$ is indepenent of $x_i$, so
\[
    \bias(\phi) = \E_{x_{\{i\}^c}} \br{e(\phi_{i'}) \E_{x_i} e(\phi_i)}
\]
Hence by the triangle inequality
\[
    |\bias(\phi)| \leq \E_{x_{\{i\}^c}} \left| \E_{x_i} e(\phi_i) \right|.
\]
But since $\phi_i$ is linear in $x_i$,
$\E_{x_i} e(\phi_i) \geq 0$,
so we may drop the absolute value signs.
Hence
\[
  |\bias(\phi)| \leq \E_{x_{I \setminus i}}  \E_{x_i} e(\phi_i(x)) = \bias(\phi_i).
\]
Repeat for every $i \in J$.
\end{proof}

\begin{lemma}
\label{bias-subadditivity}
Let $\phi, \psi: A_{[k]} \to \torus$ be multilinear.
Then
\[
  \bias(\phi + \psi) \geq \bias(\phi) \bias(\psi).
\]
\end{lemma}
\begin{proof}
Let $x, y \in A_{[k]}$ be independent.
Then
\[
  \bias(\phi) \bias(\psi)
  = \E_{x,y} e(\phi(x) + \psi(y)) = \E_{x,y} e(\phi(x) + \psi(x + y)).
\]
We may expand
\[
  \psi(x + y) = \sum_{I \subset[k]} \psi_I(x_I, y_{I^c}),
\]
where $\psi_I : A_I \times A_{I^c} \to \torus$ is again $k$-linear.
Note that
$\psi_{[k]} = \psi(x)$.
For each fixed $y \in A_{[k]}$, the previous lemma with $J = [k]$ implies that
\[
  |\E_x e(\phi(x) + \psi(x + y))|
  \leq \E_x e(\phi(x) + \psi(x))
  = \bias(\phi + \psi).
\]
Hence the lemma follows from the triangle inequality.
\end{proof}

Suppose $\phi : A_{[k]} \to \torus$ is $k$-linear
and $\psi : B_{[l]} \to \torus$ is $l$-linear.
We say $\phi$ \emph{factors through} $\psi$ if
there is an $l$-partition
\[
  [k] = I_1 \cup \cdots \cup I_l \qquad (I_1, \dots, I_l \neq \emptyset)
\]
and for each $j \in [l]$ an $|I_j|$-linear map
\[
  \phi_j : A_{I_j} \to B_j
\]
such that $\phi$ factors as
\[
  \phi = \psi(\phi_{I_1}, \dots, \phi_{I_l}),
\]
i.e.,
\[
  \phi : A_{[k]}
  \cong \prod_j A_{I_j}
  \xlongrightarrow{(\phi_{I_j})_j}
  B_{[l]}
  \xlongrightarrow{\psi}
  \torus.
\]

\begin{lemma}
    \label{bias-factor-lemma}
    Suppose $\phi : A_{[k]} \to \torus$ and $\psi : B_{[l]} \to \torus$ are multilinear and $\phi$ factors through $\psi$. Then $\bias(\phi) \geq \bias(\psi)$.
    \end{lemma}
    \begin{proof}
    Suppose $[k] = I_1 \cup \cdots \cup I_l$ and
    \[
      \phi = \psi(\phi_{I_1}, \dots, \phi_{I_l}).
    \]
    For $b \in B_{[l]}$, let
    \[
      \phi_b = \psi(\phi_{I_1} + b_1, \dots, \phi_{I_l} + b_l).
    \]
    Clearly
    \[
        \E_b \bias(\phi_b) = \bias(\psi).
    \]
    By Lemma~\ref{main-term-bias-lemma}, for every $b \in B_{[l]}$ we have
    \[
      |\bias(\phi_b)| \leq \bias(\phi_0) = \bias(\phi).
    \]
    Hence $\bias(\psi) \leq \bias(\phi)$.
\end{proof}

\begin{corollary}
    Suppose $\phi : A_{[k]} \to \torus$ and $\psi : B_{[l]} \to \torus$ are multilinear and $\phi$ factors through $\psi$.
    Then $\bias(\phi) \geq 1 - \prod_{j \neq i} (1 - 1/|B_j|)$ for each $i \in [l]$.
\end{corollary}
\begin{proof}
    Combine the previous lemma with \Cref{bias-trivial-bounds}.
\end{proof}

Recall that $m_q$ denotes the map $\Z/q \times \Z/q\Z \to \torus$ defined by $m(x, y) = xy / q \bmod 1$,
where $q$ is a prime power.

\begin{corollary}
    \label{cor:low-rank-implies-bias}
    Suppose $\phi : A_{[k]} \to \torus$ is the sum of maps $\phi_1, \dots, \phi_r$ and $\phi_i$ factors through $m_{q_i}$.
    Then
    \[
        \bias(\phi) \geq q_1^{-1} \cdots q_r^{-1}.
    \]
\end{corollary}
\begin{proof}
    Combine the previous corollary with \Cref{bias-subadditivity}.
\end{proof}

\begin{corollary}
    \label{cor:restriction}
    Suppose $A'_i \leq A_i$ for each $i$.
    Let $\phi : A_{[k]} \to \torus$ be multilinear and let $\phi'$ be the restriction of $\phi$ to $A'_{[k]}$.
    Then $\bias(\phi') \geq \bias(\phi)$.
\end{corollary}
\begin{proof}
    Since $\phi'$ factors through $\phi$, we may apply \Cref{bias-factor-lemma} to $(\phi', \phi)$.
\end{proof}

The following lemma,
which gives a little more information than the second part of Lemma~\ref{bias-trivial-bounds},
illustrates the utility of the last corollary.

\begin{lemma}
    \label{bias-exponent-lemma}
    Suppose $\phi : A_{[k]} \to \torus$ is multilinear
    and suppose the image of $\phi$ has an element of order $q = p^n$ for some prime $p$.
    Then $\bias(\phi) \leq \bias(\psi)$, where $\psi : (\Z/q\Z)^k \to \torus$ is defined by
    \[
        \psi(x_1, \dots, x_k) = x_1 \cdots x_k / q \pmod 1.
    \]
    In particular, for $k \geq 2$,
    \[
      \bias(\phi) \leq (n+1)^{k-2} / q.
    \]
\end{lemma}
\begin{proof}
    Suppose $\phi(a_1, \dots, a_k)$ has order $q = p^n$.
    Let $A'_i = \langle a_i \rangle$ for each $i$.
    Then
    \[
      \bias(\phi) \leq \bias(\phi|_{A'_{[k]}}).
    \]
    Hence without loss of generality $A_i$ is generated by $a_i$ for each $i$.
    Note that $\phi_{qa_i} = 0$ for each $i$, so $\phi$ factors through $\prod_{i=1}^k \langle a_i\rangle / \langle q a_i\rangle$.
    Hence we may assume that $qa_i = 0$ for each $i$.
    Also note that $p^{n-1}a_i \neq 0$ for each $i$, since
    \[
      \phi(\dots, a_{i-1}, p^{n-1} a_i, a_{i+1}, \dots) = p^{n-1} \phi(a_1, \dots, a_k) \neq 0.
    \]
    Hence each $a_i$ has order $q$, so $\phi$ is equivalent to the map $\psi$ in the statement of the lemma.
    
    For the second statement we may assume $\phi = \psi$.
    Let $v_p$ denote the $p$-adic valuation.
    From \eqref{bias-kernel-expression},
    \begin{align*}
          \bias(\phi)
          &= \P(x_1 \cdots x_{k-1} = 0) \\
          &= \P(v_p(x_1) + \cdots + v_p(x_{k-1}) \geq n) \\
          &\leq \sum_{v_1 + \cdots + v_{k-1} = n} \P(v_p(x_i) \geq v_i ~ \text{for each} ~ i < k) \\
          &= \sum_{v_1 + \cdots + v_{k-1} = n} p^{-v_1-\cdots-v_{k-1}} \\
          &= M / q,
    \end{align*}
    where $M$ is the number of solutions to $v_1 + \cdots + v_{k-1} = n$ in nonnegative integers.
    Clearly $M \leq (n+1)^{k-2}$.
\end{proof}

\section{The structure theorem}

In this section we prove \Cref{thm:bias-inverse-theorem}.
Throughout we consider $k$ to be fixed and dependence on $k$ will not be explicit.
We start with some simple reductions.

Suppose $\phi : A_{[k]} \to \torus$ is multilinear and $\bias(\phi) \geq \eps$.
We may decompose each $A_i$ into its $p$-primary parts:
\[
  A_i = \bigoplus_p A_i^{(p)}.
\]
Then $\phi$ decomposes as
\[
  \phi = \sum_p \phi_p,
\]
where $\phi_p$ factors as
\[
  A_{[k]} \to A_1^{(p)} \times \cdots \times A_k^{(p)} \to \torus.
\]
The bias of $\phi$ correspondingly decomposes as
\[
  \bias(\phi) = \prod_p \bias(\phi_p).
\]
By \Cref{bias-trivial-bounds}, $\bias(\phi_p) \leq 1 - (1 - 1/p)^{k-1} \leq 1 - 1/2^{k-1}$ whenever $\phi_p$ is nontrivial,
so since $\bias(\phi) \geq \eps$ it follows that there are at most $\log \eps^{-1} / \log(1 - 1/2^{k-1})^{-1}$ primes such that $\phi_p \not \equiv 0$.
Hence it suffices to consider the $p$-primary parts separately, so without loss of generality $A_1, \dots, A_k$ are $p$-groups.

Let $K_i = \{a_i \in A_i : \phi_{a_i} \equiv 0\}$.
By replacing $A_i$ with $A_i / K_i$ we may assume $K_i = 0$,
so $\phi$ is nondegenerate in the sense that $\phi_{a_i} \not\equiv 0$ for each nonzero $a_i \in A_i$.
Then if $A_i$ has an element of order $q = p^n$, the image of $\phi$ also contains an element of order $q$, so by \Cref{bias-exponent-lemma},
\[
  \bias(\phi_p) \leq (n+1)^{k-2} / q \leq (\log_2 q + 1)^{k-2} / q.
\]
This implies that $q$ is bounded in terms of $\eps$ (and $k$).
Hence we may assume that $A_1, \dots, A_k$ have bounded exponent (and in particular $p$ is bounded).

\def\agemO{\mho}
We will now deduce \Cref{thm:bias-inverse-theorem} from \Cref{thm:AR-PR-field-case} using the following elementary lemma about extending multilinear maps.
For any abelian $p$-group, let $A[p]$ denote the subgroup $\{x \in A : p x = 0\}$ and let $pA$ denote the subgroup $\{px : x \in A\}$.

\begin{lemma}
    \label{lem:extension}
    Let $A_1, \dots, A_k$ be finite abelian $p$-groups and $q$ a power of $p$.
    \begin{enumerate}[(1)]
        \item (domain enlargement)
        Suppose $\phi : pA_1 \times A_{[2, k]} \to \Z/q\Z$ is a multilinear map.
        Assume $\phi(x) = 0$ whenever $px_i = 0$ for any $i > 1$.
        Then $\phi$ factors through $pA_1 \times \prod_{i \in [2, k]} A_i / A_i[p]$,
        and there is a multilinear map $\psi : A_{[k]} \to \Z/(pq)\Z$ extending $\phi$ in the sense that the following diagram commutes:
        \[
            \begin{tikzcd}
                & pA_1 \times \prod_{i \in [2, k]} A_i / A_i[p] \arrow{dr}{} \\
                pA_1 \times A_{[2, k]} \arrow{rr}{\phi} \arrow{ur}{} \arrow{d}{}
                && \Z/q\Z \arrow{d}{\times p} \\
                A_{[k]} \arrow[dashed]{rr}{\psi}
                && \Z/(pq)\Z
            \end{tikzcd}.
        \]
        \item (range enlargement)
        Suppose $\phi : A_{[k]} \to \Z/q\Z$ is a multilinear map such that $\phi(x) = 0$ whenever $p x_i = 0$ for any $i \in [k]$.
        Then $\phi$ factors through $\prod_{i \in [k]} A_i / A_i[p]$, and there is a multilinear map $\psi : A_{[k]} \to \Z/(pq)\Z$ such that $\psi \bmod q = \phi$,
        as in the following commutative diagram:
        \[
            \begin{tikzcd}
                & \prod_{i \in [k]} A_i / A_i[p] \ar{dr}{} \\
                A_{[k]} \ar{ur} \ar{rr}{\phi} \ar[dashed]{drr}{\psi}
                && \Z/q\Z \\
                && \Z/(pq)\Z \ar[swap]{u}{\bmod q}
            \end{tikzcd}.
        \]
        \item (extension of rank-one maps)
        Suppose $\phi : pA_1 \times A_{[2,k]} \to \torus$ is a multilinear map such that $\phi(x) = 0$ whenever $px_i = 0$ for any $i > 1$, and factoring through $m_q$.
        Then $\phi$ extends to a multilinear map $\psi : A_{[k]} \to \torus$ factoring through $m_{pq}$,
        as in the following commutative diagram:
        \[
            \begin{tikzcd}
                & pA_1 \times \prod_{i \in [2, k]} A_i / A_i[p] \arrow{dr}{} \\
                pA_1 \times A_{[2, k]} \arrow{rr}{\phi} \arrow{ur}{} \arrow{d}{}
                && (\Z/q\Z)^2 \arrow{r}{m_q}
                & \torus \\
                A_{[k]} \arrow[dashed]{rr}{\psi}
                && (\Z/(pq)\Z)^2 \arrow[swap]{ur}{m_{pq}}
            \end{tikzcd}.
        \]
    \end{enumerate} 
\end{lemma}
\begin{proof}
    By the structure theorem for finite abelian groups, we may write
    \[
    A_i = \langle e_{i1} \rangle \oplus \cdots \oplus \langle e_{id_i} \rangle \qquad (1 \leq i \leq k),
    \]
    where each $e_{ij}$ has order some power of $p$.
    
    (1)
    Write
    \[
        \phi_{j_1 \cdots j_k} = \phi(p e_{1j_1}, e_{2j_2}, \dots, e_{kj_k}) \in \Z/q\Z.
    \]
    We may identify $\Z/q\Z$ with $p(\Z/(pq)\Z)$.
    For each $j_1, \dots, j_k$, let $\psi_{j_1 \cdots j_k} \in \Z/(pq)\Z$ be an arbitrary solution to
    \[
        p \psi_{j_1 \cdots j_k} = \phi_{j_1 \cdots j_k}.
    \]
    Suppose $i > 1$ and $e_{ij_i}$ has order $p^n$.
    By the assumption that $\phi$ factors through $pA_1 \times \prod_{i \in [2, k]} A_i / A_i[p]$,
    we must have $p^{n-1} \phi_{j_1\cdots j_k} = 0$.
    Hence $p^n \psi_{j_1\cdots j_k} = 0$.
    Similarly, if $e_{1j_1}$ has order $p^n$, then
    \[
        p^n \psi_{j_1\cdots j_k} = p^{n-1} \phi_{j_1 \cdots j_k} = \phi(p^n e_{1j_1}, e_{2j_2}, \dots, e_{kj_k}) = 0.
    \]
    Hence we may define $\psi : A_{[k]} \to \Z/(pq)\Z$ by linearly extending the definition
    \[
        \psi(e_{1j_1}, \dots, e_{k j_k}) = \psi_{j_1 \cdots j_k}.
    \]
    The maps $\psi |_{pA_1 \times A_{[2, k]}}$ and $p \phi$ agree on generators, so they are the same.
    
    (2)
    Write
    \[
        \phi_{j_1 \cdots j_k} = \phi(e_{1j_1}, e_{2j_2}, \dots, e_{kj_k}) \in \Z/q\Z.
    \]
    For each $j_1, \dots, j_k$, let $\psi_{j_1 \cdots j_k} \in \Z/(pq)\Z$ be an arbitrary solution to
    \[
        \psi_{j_1 \cdots j_k} \equiv \phi_{j_1 \cdots j_k} \pmod q.
    \]
    Suppose $p^n$ is the order of one of the generators $e_{ij_i}$.
    Then $p^{n-1} \phi_{j_1 \cdots j_k} = 0 \pmod q$, so $p^n \psi_{j_1 \cdots j_k} = 0 \pmod {pq}$.
    Hence we may define $\psi : A_{[k]} \to \Z/(pq)\Z$ by linearly extending the definition
    \[
        \psi(e_{1j_1}, \dots, e_{k j_k}) = \psi_{j_1 \cdots j_k}.
    \]
    Then $\psi \bmod q$ and $\phi$ agree on generators, so they are the same.
    
    (3)
    By assumption $\phi = m_q(\phi_1, \phi_2)$ where $\phi_1$ is defined on $pA_1 \times A_I$ for some $I \subsetneq [2,k]$ and $\phi_2$ is defined on $A_{[2, k] \setminus I}$.
    By (1), $\phi_1$ may be extended to $\psi_1 : A_1 \times A_I \to \Z/(pq)\Z$ in such a way that $\psi_1 |_{p A_1 \times A_I} = p\phi_1$.
    By (2), $\phi_2$ may be extended to $\psi_2 : A_{[2, k] \setminus I} \to \Z/(pq)\Z$ in such a way that $\psi_2 \bmod q = \phi_2$.
    Let $\psi = m_{pq}(\psi_1, \psi_2)$.
    Then for $x \in p A_1 \times A_{[2, k]}$ we have
    \[
        \psi(x) = \psi_1(x) \psi_2(x) / (pq) \bmod 1 = \phi_1(x) \phi_2(x) / q \bmod 1 = \phi(x),
    \]
    so $\psi$ extends $\phi$.
\end{proof}

We can now prove \Cref{thm:bias-inverse-theorem}.

\begin{proof}[Proof of \Cref{thm:bias-inverse-theorem}]
Let the exponent of $A_i$ be $p^{n_i}$.
We will prove the theorem by induction on $n_1 + \cdots + n_k$.
If $n_i = 1$ for each $i$ then each $A_i$ is elementary abelian,
hence a vector space over $\F_p$.
In this case the result follows from \Cref{thm:AR-PR-field-case}.
Hence assume $n_1 > 1$ without loss of generality.
Then $pA_1$ and $A_1 / p A_1$ both have smaller exponent than $A_1$.

The restriction $\phi_1$ of $\phi$ to $pA_1 \times A_{[2,k]}$ factors through a map $\phi_1'$ defined on $pA_1 \times \prod_{i \in [2, k]} A_i / A_i[p]$, as in the following diagram:
\[
    \begin{tikzcd}
        pA_1 \times A_{[2, k]} \ar{d}{} \ar{drr}{\phi_1} \ar{rr} && pA_1 \times \prod_{i \in [2, k]} A_i / A_i[p] \ar{d}{\phi_1'} \\
        A_{[k]} \ar{rr}{\phi} && \torus
    \end{tikzcd}.
\]
By \Cref{cor:restriction}, $\bias(\phi'_1) = \bias(\phi_1) \geq \eps$,
so by induction
\[
    \phi'_1 = \psi_1' + \cdots + \psi_r'
\]
for $r = O_\eps(1)$ and some
\[
    \psi_1', \dots, \psi_r' : pA_1 \times \prod_{i \in [2, k]} A_i / A_i[p] \to \torus
\]
where each $\psi_i'$ factors through $m_{q_i}$ for some $q_i \leq O_\eps(1)$.
By \Cref{lem:extension}, each $\psi'_i$ extends to a multilinear map $\psi_i : A_{[k]} \to \torus$
factoring through $m_{pq_i}$.
Let
\[
    \phi_2 = \phi - (\psi_1 + \cdots + \psi_r).
\]
Then $\phi_2$ also has bias bounded away from zero (by \Cref{bias-subadditivity} and \Cref{cor:low-rank-implies-bias}),
and $\phi_2$ factors through $(A_1 / pA_1) \times A_{[2,k]}$.
By induction again there are multilinear maps $\psi_{r+1}, \dots, \psi_{r_2} : (A_1/pA_1) \times A_{[2,k]} \to \torus$,
where $r_2 = O(1)$,
such that each $\psi_i$ factors through $m_q$ for some bounded $q$,
and such that
\[
  \phi_2 = \psi_{r+1} + \cdots \psi_{r_2}.
\]
Hence
\[
  \phi = \psi_1 + \cdots + \psi_{r_2},
\]
and the induction is complete.
\end{proof}

We end by deducing \Cref{cor:main-corollary}.

\begin{proof}[Proof of \Cref{cor:main-corollary}]
    Let $A_k = \hat B$ and let $\phi(x) = x_k(F(x_1, \dots, x_{k-1}))$.
    Then $\phi$ is multilinear and $\bias(\phi) = \P(f = 0) \geq \eps > 0$,
    so \Cref{thm:bias-inverse-theorem} implies that $\phi$ is the sum of $O_\eps(1)$ multilinear functions of the form
    \[
        \psi(x) = m_q( \phi_1(x_I), \phi_2(x_{I^c})),
    \]
    where $0 < |I| < k$ and $q \leq O_\eps(1)$.
    Assume without loss of generality that $k \in I^c$.
    Since $\psi$ is linear in $x_k$ and $\hat A_k \cong B$, there is some $b = b(x_{[k-1]}) \in B$ such that
    \[
        \psi(x) = x_k(b(x_{[k-1]})).
    \]
    Moreover, $b$ must be multilinear in $x_{[k-1]}$, and $b$ can depend on $x_I$ only through $\phi_1(x_I)$.
    Hence $\psi$ has the form
    \[
        \psi(x) = x_k(G(g(x_I), x_{[k-1] \setminus I})),
    \]
    where $g = \phi_1$.
    This implies that $F$ is the sum of $O_\eps(1)$ maps of the form $G(g(x_I), x_{[k-1] \setminus I})$.
    
    Now fix $I$ and suppose the terms of the form $G(g(x_I), x_{[k-1] \setminus I})$ are
    \[
        G_j(g_j(x_I), x_{[k-1] \setminus I}) \qquad (1 \leq j \leq n).
    \]
    Let $g = (g_1, \dots, g_n)$, and note that $|\cod(g)|$ is still bounded and $G_j(g_j(x_I), x_{[k-1] \setminus I}) = G'_j(g(x_I), x_{[k-1] \setminus I})$,
    where $G'_j$ is the composite of $G_j$ and the projection $\pi_j$.
    Let $G = G'_1 + \cdots + G'_n$. Then
    \[
        \sum_{j=1}^n G_j(g_j(x_I), x_{[k-1] \setminus I}) = G(g(x_I), x_{[k-1] \setminus I}).
    \]
    Hence $F$ has the claimed form.
\end{proof}

\section{Possible biases}

Let $\Phi_k$ be the set of all multilinear maps $\phi : A_{[k]} \to \torus$ for any finite abelian groups $A_1, \dots, A_k$.
In this section we probe the set of all possible biases
\[
    B_k = \{ \bias(\phi) : \phi \in \Phi_k \}.
\]
For example,
\begin{align*}
    B_1 = \{0, 1\}, &&
    B_2 = \{1 / n : n \geq 1\}.
\end{align*}
By \eqref{bias-kernel-expression}, $B_k \subset [0, 1]$.
We will show that \Cref{thm:bias-inverse-theorem} implies that $B_k$ is a small subset of $[0, 1]$ in various senses.
The results of this section are analogous to those proved in \cite{eberhard} essentially for $B_3$ and for the set of commuting probabilities of finite groups.

A slight generalization is useful. Call $\phi : A_{[k]} \to \torus$ \emph{multiaffine} if $\phi(x)$ is affine-linear in each variable $x_i$ separately.
Equivalently, $\phi$ is multiaffine if and only if it can be written
\[
    \phi = \sum_{I \subset [k]} \phi_I,
\]
where $\phi_I$ is a multilinear map $A_I \to \torus$.
The greatest $d = |I|$ such that $\phi_I$ is nontrivial is the \emph{degree} of $\phi$.
Let $\Phi_{k, d}$ be the set of all multiaffine maps $\phi:A_{[k]} \to \torus$ (for any $A_1, \dots, A_k$) of degree at most $d$ and zero constant term,
and let
\[
    B_{k, d} = \{ \bias(\phi) : \phi \in \Phi_{k,d} \}.
\]
Thus $B_{k, d}$ is a subset of the unit disc $D = \{z \in \C: |z| \leq 1\}$.

\begin{example}[Gauss sums]
    Let $A_1 = A_2 = A_3 = \F_p$, where $p$ is an odd prime, and let
    \[
        \phi(x, y, z) = (xy + xz + yz) / p \bmod 1 \qquad (x, y, z \in \F_p)
    \]
    Thus $\phi \in \Phi_{3, 2}$.
    Then
    \[
        \bias(\phi) = \E_{x, y, z \in \F_p}\, e((xy + xz + yz)/p) = p^{-2} \bar{G(p)},
    \]
    where $G(p)$ is the Gauss sum
    \[
        G(p) = \sum_{x \in \F_p} e(x^2 / p).
    \]
    It was proved by Gauss that $G(p) = p^{1/2}$ or $ip^{1/2}$ according to whether $p$ is 1 or 3 mod 4.
    Hence $B_{3,2}$ includes $i^{\pm (p-1)^2/4} p^{-3/2}$ for every odd prime $p$.
\end{example}

If $X$ is a subset of a compact metric space, let $X'$ be the set of limit points of $X$.
We iterate this operation transfinitely according to the following rules:
\begin{align*}
    &X^{(0)} = X, \\
    &X^{(\alpha+1)} = (X^{(\alpha)})' && (\alpha~\text{an ordinal}), \\
    &X^{(\lambda)} = \bigcap_{\alpha < \lambda} X^{(\alpha)} && (\lambda~\text{a nonzero limit ordinal}).
\end{align*}
If $X$ is countable then there is a countable ordinal $\beta$ such that $X^{(\beta)} = \emptyset$.
The smallest such $\beta$ must be a successor ordinal, say $\alpha+1$, and
this ordinal $\alpha$ is called the \emph{Cantor--Bendixson rank} of $X$ and denoted $\CB(X)$.
It is the unique ordinal such that $X^{(\alpha)}$ is finite and nonempty.%
\footnote{More often $\beta$ itself is called the Cantor--Bendixson rank, but the present usage is a common variant for countable sets.}

\begin{theorem}
The set $B_{k,d}$ has the following smallness properties.
\begin{enumerate}[(1)]
    \item\label{item:algebraic-limit-points} All limit points of $B_{k, d}$ are algebraic (in fact, rational linear combinations of roots of unity).
    \item\label{item:CB-bound} $\CB(B_{k, d}) \in \{1, \omega, \dots, \omega^{d-2}\}$ for $d \geq 2$.
\end{enumerate}
\end{theorem}

In particular, \Cref{item:algebraic-limit-points} shows that $\bar{B_{k,d}}$ is countable, so in particular it is nowhere dense and measure-zero.

\begin{proof}
    For $\eps > 0$ let
    \[
        B_{k,d,\eps} = B_{k,d} \cap \{z \in D : |z| \geq \eps\}.
    \]
    Let $\phi \in \Phi_{k, d, \eps}$, so $\phi : A_{[k]} \to \torus$ is a multiaffine map of degree $d$ and zero constant term defined on some finite abelian groups $A_1, \dots, A_k$ and $|\bias(\phi)| \geq \eps$.
    Write
    \[
        \phi = \sum_{I \subset [k]} \phi_I,
    \]
    where $\phi_I : A_I \to \torus$ is multilinear.
    Let $\phi_d$ be the sum of the terms with $|I| = d$. By \Cref{main-term-bias-lemma}, for each $I \subset [k]$ of size $d$ we have
    \[
        \bias(\phi_I) \geq |\bias(\phi)| \geq \eps,
    \]
    so by \Cref{thm:bias-inverse-theorem} there is an expression for $\phi_I$ as the sum of $O_\eps(1)$ maps factoring through $m_q$ for some bounded prime power $q$.
    Hence there is an expression for $\phi_d$ of the form
    \[
        \phi_d = \sum_{i=1}^r m_{q_i}(\phi_i^L, \phi_i^R),
    \]
    where $r = O_\eps(1)$ and, for each $i$, $\phi_i^L$ and $\phi_i^R$ ($1 \leq i \leq r$) are multilinear maps with disjoint sets of variables and codomain $\Z/q_i\Z$, where $q_i = O_\eps(1)$.
    Hence
    \[
        e(\phi_d) = F(\phi_1^L, \dots, \phi_r^L, \phi_1^R, \dots, \phi_r^R),
    \]
    where
    \[
        F(x_1, \dots, x_r, y_1, \dots, y_1) = e(x_1 y_1 / q_1 + \cdots + x_r y_r / q_r).
    \]
    Let $\Gamma = \br{\prod_{i=1}^r \Z/q_i \Z}^2$ and let $Q = q_1 \cdots q_r = |\Gamma|^{1/2}$. The Fourier expansion of $F$ is
    \[
        F(x_1, \dots, x_r, y_1, \dots, y_s) = \frac1Q \sum_{(u, v) \in \Gamma} e\br{\sum_{i=1}^r (-u_i v_i + u_i x_i + v_i y_i) / q_i}.
    \]
    Hence
    \[
        e(\phi_d) = \frac1Q \sum_{(u, v) \in \Gamma} e\br{ \sum_{i=1}^r (-u_i v_i + u_i \phi_i^L + v_i \phi_i^R) / q_i}
    \]
    and
    \[
        \bias(\phi) = \frac1Q \sum_{(u, v) \in \Gamma} e(-u_1v_1/q_1 - \cdots - u_r v_r / q_r) \bias(\psi_{u,v}),
    \]
    where
    \[
        \psi_{u, v} = \sum_{i=1}^r (u_i \phi_i^L + v_i \phi_i^R) / q_i + (\phi - \phi_d).
    \]
    Note that $\psi_{u, v}$ is multi-affine of degree at most $d-1$ with zero constant term. Since $|\Gamma| \leq O_\eps(1)$, we have proved that
    \begin{equation}
        \label{eq:Bkd-recursion}
        B_{k, d, \eps} \subset O_\eps(1) H_{O_\eps(1)} B_{k, d-1},
    \end{equation}
    where $H_X$ is the (finite) set of all complex numbers of the form $\zeta / m$ where $\zeta$ is a root of unity of order at most $X$
    and $m$ is a positive integer of size at most $X$.
    Here we are also using the standard notation $ST = \{s t : s \in S, t \in T\}$ for the product of sets $S$ and $T$ and $nS = S + \cdots + S$ for the sum of $n$ copies of a set $S$.
    
    We will now use induction on $d$ to prove \Cref{item:algebraic-limit-points,item:CB-bound}.
    As the base of the induction, note that an affine linear map with zero constant term is just a linear map,
    so $B_{k,1} = B_1 = \{0, 1\}$.
    Hence, for $d = 2$, \eqref{eq:Bkd-recursion} shows that $B_{k, 2, \eps}$ is finite for all $\eps > 0$, so $B'_{k, 2} \subset \{0\}$.
    Hence we may assume $d >2 $.
    
    Let $z$ be a limit point of $B_{k, d}$. Then either $z = 0$ or $z$ is a limit point of $B_{k, d, \eps}$ for $\eps = |z|/2$.
    In the latter case \eqref{eq:Bkd-recursion} shows that
    \[
        z \in O_\eps(1) H_{O_\eps(1)} \bar{B_{k, d-1}}.
    \]
    By induction on $d$ this shows that $z$ is a rational linear combination of roots of unity.
    This proves \Cref{item:algebraic-limit-points}.
    
    Let $\alpha = \CB(B_{k, d})$. Then it follows from \eqref{eq:Bkd-recursion} that
    \[
        B_{k,d,\eps}^{(\alpha O_\eps(1))} = \emptyset.
    \]
    Since this holds for all $\eps > 0$, it follows that
    \[
        B_{k, d}^{(\alpha \omega)} \subset \{0\}.
    \]
    Hence
    \[
        \CB(B_{k, d}) \leq \CB(B_{k, d-1}) \omega,
    \]
    and this proves $\CB(B_{k, d}) \leq \omega^{d-2}$ by induction.
    Finally, since $B_{k,d}$ is closed under multiplication, the argument of \cite{eberhard}*{Lemma~5.1} shows that $\CB(B_{k,d}) = \omega^\gamma$ for some ordinal $\gamma$, so $\CB(B_{k,d}) \in \{1, \omega, \dots, \omega^{d-2}\}$.
\end{proof}

Since $B_k \subset B_{k,k}$, the theorem applies in particular to $B_k \subset [0, 1]$ and shows
\begin{enumerate}[(1)]
    \item all limit points of $B_k$ are algebraic,
    \item $\CB(B_k) \in \{1, \omega, \dots, \omega^{k-2}\}$.
\end{enumerate}
It is very likely that all limit points of $B_k$ are actually rational, but a slightly more sophisticated analysis is required to prove this.
As to the Cantor--Bendixson rank, basic examples demonstrate that $\CB(B_k) > 1$ for $k > 2$, so
\[
    \CB(B_1) = 0, \qquad \CB(B_2) = 1, \qquad \CB(B_3) = \omega,
\]
but the value of $\CB(B_k)$ for $k > 3$ is unclear.
Relatedly, it is plausible that $B_k$ is well-ordered by the reverse of the natural order, i.e., for every $x \in B_k$ there is some $\eps > 0$ such that $(x-\eps, x) \cap B_k = \emptyset$; if this is the case it follows that $B_k$ has order type $\omega^{\CB(B_k)}$.
We leave these questions as open problems.

\begin{problem}
    \leavevmode
    \begin{enumerate}[(1)]
        \item Prove that all limit points of $B_k$ are rational.
        \item Determine the Cantor--Bendixson rank of $B_k$. In particular, determine whether $\CB(B_4)$ is $\omega$ or $\omega^2$.
        \item Determine whether $B_k$ is reverse-well-ordered.
    \end{enumerate}
\end{problem}

\bibliography{refs}
\bibliographystyle{plain}

\end{document}